\documentclass[12pt]{amsart}
\usepackage{graphicx}
\usepackage{amsmath}
\usepackage{amssymb}
\usepackage{amsfonts}
\usepackage{verbatim}
\usepackage{scalefnt}
\usepackage{multirow}
\usepackage{enumerate}
\usepackage{mathtools}
\usepackage{float}
\usepackage{enumitem}
\restylefloat{figure}
\usepackage[margin=3.5cm]{geometry}

\usepackage{fancyhdr}
\usepackage{hyperref}

\begin{document}

\def\COMMENT#1{}

\newcommand{\case}[1]{\medskip\noindent{\bf Case #1} }
\newcommand{\step}[1]{\medskip\noindent{\bf Step #1} }
\newcommand{\fl}[1]{\lfloor #1\rfloor}
\newcommand{\ceil}[1]{\lceil #1\rceil}
\newtheorem{problem}{Problem}
\newtheorem{theorem}{Theorem}
\newtheorem{lemma}[theorem]{Lemma}
\newtheorem{proposition}[theorem]{Proposition}
\newtheorem{corollary}[theorem]{Corollary}
\newtheorem{conjecture}[theorem]{Conjecture}
\newtheorem{claim}[theorem]{Claim}
\newtheorem{definition}[theorem]{Definition}
\newtheorem*{definition*}{Definition}
\newtheorem{fact}[theorem]{Fact}
\newtheorem{observation}[theorem]{Observation}
\newtheorem{question}[theorem]{Question}
\newtheorem{remark}[theorem]{Remark}
\newcommand{\ind}{\mathrm{ind}}
\newcommand{\mis}{\mathrm{mis}}
\newcommand{\MIS}{\mathrm{MIS}}
\newcommand{\fm}{f_{\max}}
\newcommand{\msf}{\mathrm{MSF}}

\numberwithin{equation}{section}
\numberwithin{theorem}{section}

\newcommand{\De}{\Delta}
\newcommand{\de}{\delta}
\newcommand{\Ga}{\Gamma}
\newcommand{\ep}{\varepsilon}
\newcommand{\cP}{\mathcal{P}}
\newcommand{\cI}{\mathcal{I}}
\newcommand{\cG}{\mathcal{G}}
\newcommand{\cT}{\mathcal{T}}
\newcommand{\cF}{\mathcal{F}}
\newcommand{\cL}{\mathcal{L}}
\newcommand\ex{\ensuremath{\mathrm{ex}}}
\newcommand{\eul}{e}
\newcommand{\pr}{\mathbb{P}}
\newcommand{\bE}{\mathbb{E}}
\newcommand{\bN}{\mathbb{N}}
\newcommand{\bZ}{\mathbb{Z}}


\newif\ifnotesw\noteswtrue
\newcommand{\notes}[1]{\ifnotesw $\blacktriangleright$\ {\sf #1}\ 
	$\blacktriangleleft$ \fi}

\title[Groups with few maximal sum-free sets]{Groups with few maximal sum-free sets}
\author{Hong Liu and Maryam Sharifzadeh}
\thanks{H.~L.\ was supported by the Leverhulme Trust Early Career Fellowship~ECF-2016-523. M.~Sh.\ was supported by the European Union’s Horizon 2020 research and innovation programme under the Marie Curie Individual Fellowship agreement No 752426.}

\begin{abstract}
	We show that, in contrast to the integers setting, almost all even order abelian groups $G$ have exponentially fewer maximal sum-free sets than $2^{\mu(G)/2}$, where $\mu(G)$ denote the size of a largest sum-free set in $G$. This confirms a conjecture of Balogh, Liu, Sharifzadeh and Treglown.
\end{abstract}

\date{\today}
\maketitle

\section{Introduction}
A triple $\{x,y,z\}$ of natural numbers is a \emph{Schur triple} if $x+y=z$\footnote{Note that $x$ and $y$ are not necessarily distinct.}. A set $S\subseteq \bN$ is \emph{sum-free} if it does not contain any Schur triple, in other words, $(S+S)\cap S=\emptyset$. Sum-free set is a fundamental notion in combinatorial number theory. Its study dates back to the classical Schur's theorem~\cite{Sch16} in Ramsey theory in 1916, which states that any finite colouring of $\bN$ contains infinitely many monochromatic Schur triples. If we look at sum-free sets in the first $n$ integers $[n]:=\{1,\ldots,n\}$, it is easy to see that $\mu(n)$, the size of a largest sum-free subset of $[n]$, is $\ceil{n/2}$. Both the set of odd integers and $\{\fl{n/2}+1,\ldots n\}$ are examples of extremal sets. Denote by $f(n)$ the number of sum-free subsets of $[n]$. Since all subsets of a sum-free set are also sum-free, we have that $f(n)\ge 2^{\mu(n)}$. Cameron and Erd\H os~\cite{CameronErdos-SF1} conjectured that this trivial lower bound in fact gives the correct order, that is, $f(n)=O(2^{\mu(n)})$. Their conjecture was only proven more than a decade later independently by Green~\cite{G-CE} and Sapozhenko~\cite{sap}, both of which proved the stronger statement that there are two constants $C_0$ and $C_1$ such that $f(n)=(C_i+o(1))2^{\mu(n)}$, for $n\equiv i\mod 2$. 

Let us consider now a subcollection of ``largest'' sum-free sets. A sum-free set  $S\subseteq [n]$ is \emph{maximal} if it is not contained in any larger sum-free subset of $[n]$, and denote by $\fm(n)$  the number of maximal sum-free subsets of $[n]$. Motivated by the fact that all the sum-free sets in the above trivial lower bound for $f(n)$ lie in two maximal ones, Cameron and Erd\H os~\cite{CameronErdos-SF2} raised the question of enumerating maximal sum-free subsets of $[n]$. In particular, they asked whether $\fm(n)$ is exponentially smaller than $f(n)$. In the same paper, they showed that $\fm(n)\ge 2^{\mu(n)/2}$. Recently, Balogh, Treglown, and the authors~\cite{BLST2} gave an exact answer to this question, showing that, there exist constants $C_i$, $i\in [4]$, such that $\fm(n)=(C_i+o(1))2^{\mu(n)/2}$ for $n\equiv i\mod 4$.

A natural direction is to consider these questions for abelian groups. For an abelian group $G$,  we can define $\mu(G)$, $f(G)$, and $\fm(G)$ analogous to the integers setting.  Interest in sum-free subsets of abelian groups goes back to the 1960s. Estimating $\mu(G)$ turns out to be a much more difficult task. It was not until 2005 that Green and Ruzsa~\cite{GR-g} determined $\mu(G)$ for all finite abelian groups $G$. They also proved that, analogous to the integers setting, $f(G)=2^{(1+o(1))\mu(G)}$. One can then ask the question similar to Cameron and Erd\H os's: Is $\fm(G)$ exponentially smaller than $f(G)$?  Wolfovitz~\cite{wolf} showed that this is indeed the case for all even order abelian groups. This was extended to all abelian groups by Balogh, Treglown, and the authors~\cite{BLST2}, in particular,
\begin{equation}\label{eq-g}
	\fm(G)\le 3^{(1/3+o(1))\mu(G)}.
\end{equation}
Considering $\fm (n)=\Theta(2^{\mu(n)/2})$, the following question was raised in~\cite{BLST2}, asking whether analogous bound holds for abelian groups.
\begin{question}\label{ques-mu-over-2}
	Given an abelian group $G$, is it true that $f_{\max}(G)\le 2^{(1/2+o(1))\mu(G)}$?
\end{question} 
This stronger bound holds (\cite{BLST2}) for the group $\bZ_2^k:=\bZ_2\oplus\cdots\oplus \bZ_2$: $f_{\max}(\bZ_2^k)=2^{(1/2+o(1))\mu(\bZ_2^k)}$. 

It was also suspected in~\cite{BLST2} that there is an infinite class of abelian groups for which the upper bound in Question~\ref{ques-mu-over-2} is far from tight.
\begin{conjecture}\label{conjnew2}
	There exists a sequence of finite abelian groups $\{G_i\}_{i\in \bN}$ of increasing order such that for all $i$, $f_{\max} (G_i)$ is exponentially smaller than $2^{\mu (G_i)/2}.$
\end{conjecture}

We confirm this conjecture, showing that, somewhat surprisingly for almost all even order groups $\fm(G)$ is substantially smaller compared to the integers setting. 

\begin{theorem}\label{thm-evengroups}
	There exist a constant $c>0$ and an integer $n_0$ such that for almost all even order groups $G$ with $|G|>n_0$,
	\begin{align}\label{eq-exp-small}
	f_{\max}(G)\le 2^{(1/2-c)\mu(G)}.
	\end{align}
\end{theorem}

A more formal statement will be given in Section~\ref{sec-almostallevens}. We remark that the constant $c$  can be taken as for instance $10^{-4}$.
Our result suggests that $\bZ_2^k$ might be the only exception among all even order groups achieving the bound $2^{(1/2+o(1))\mu (G)}$. We will discuss more on this in the concluding remarks. Our proof can be extended to find more groups satisfying the bound in~\eqref{eq-exp-small}.

\begin{theorem}\label{thm-other-type1}
	Let $G$ be $\bZ_2\oplus\cdots\oplus \bZ_2\oplus \bZ_4$ or $\bZ_5\oplus H$ with $|H|$ odd. Then $f_{\max} (G)$ is exponentially smaller than $2^{\mu (G)/2}.$
\end{theorem}

On the other hand, we show that the bound in~\eqref{eq-g} cannot be improved, giving a negative answer to Question~\ref{ques-mu-over-2}.
\begin{proposition}\label{prop-9}
	Let $G$ be an abelian group of order $n$, such that $9|n$. Then 
	\begin{align*}
		f_{\max}(G)=3^{(1/3+o(1))\mu(G)}.
	\end{align*}
\end{proposition}

The rest of the paper is organised as follows. In Section~\ref{sec-prelim}, we introduce all the tools and useful results. Then, we prove  Theorem~\ref{thm-evengroups} in Section~\ref{sec-almostallevens}. The proof of Theorem~\ref{thm-other-type1} will be given in Section~\ref{sec-type1}. The proof of Proposition~\ref{prop-9} will be given in Section~\ref{sec-type2}. Some concluding remarks and open problems will be given in Section~\ref{sec-rmk}.


\section{Preliminaries}\label{sec-prelim}
\subsection{Notation}
For a graph $\Ga$, we write $V(\Ga)$ and $E(\Ga)$ for the set of vertices and edges of $\Ga$, respectively. We allow at most one loop at each vertex. Denote $e(\Ga):=|E(\Ga)|$ and $v(\Ga):=|V(\Ga)|$.
Given $x \in V(\Ga)$, we write $N(x, \Ga)$ for the set of vertices adjacent to $x$ in $\Ga$. We also define $d(x,\Ga):=|N(x,\Ga)|$. Note that a loop at $x$ contributes two to the degree of $x$. We write $\de (\Ga)$ for the minimum degree and $\De (G)$ for the maximum degree of $\Ga$.
Denote by $\Ga[T]$ the induced subgraph of $\Ga$ on the vertex set $T$, and $\Ga \setminus T$ the induced subgraph of $\Ga$ on the vertex set $V(\Ga)\setminus T$. For $E_1\subseteq E(\Ga)$, define $\Ga\setminus E_1\subseteq \Ga$ to be the subgraph on the same vertex set with $E(\Ga\setminus E_1)=E(\Ga)\setminus E_1$.

Throughout the paper, unless otherwise stated, all groups are finite and abelian and all logarithms are on base 2. We omit floors and ceilings where the argument is unaffected. 

\subsection{Number theoretic tools}\label{sec-gp-aux}
\begin{definition}\label{defn-group-types}
	Let $G$ be an abelian group of order $n$. 
	\begin{itemize}
		\item If $n$ is divisible by a prime $p \equiv 2$ (mod 3), then we say that $G$ is \emph{type~I($p$)}, for smallest such $p$. 
		
		\item If $n$ is not divisible by any prime $p  \equiv 2$ (mod 3), but $3|n$, then we say that $G$ is \emph{type~II}.
		
		\item Otherwise, $G$ is \emph{type~III}.
	\end{itemize}
\end{definition}

The following theorem was proved for type I and II groups by Diananda and Yap~\cite{DianandaYap}. Later, it was proved for some special type III groups~(see~\cite{RS-SumFreeGroup,Yap-SumFreeGroup1,Yap-SumFreeGroup2}), and for all type~III groups in~\cite{GR-g}. 
\begin{theorem}\label{thm-MSF-Groups}
	For all finite abelian groups $G$, if $G$ is type I($p$) then $\mu(G)=\frac{1}{3}+\frac{1}{3p}$. Otherwise, if $G$ is type II then $\mu(G)=\frac{1}{3}$. Finally, if $G$ is type III then $\mu(G)=\frac{1}{3}-\frac{1}{3m}$, where $m$ is the exponent (largest order of any element) of $G$.
\end{theorem}

We will use the following result by Green and Ruzsa~\cite{GR-g} on the structure of large sum-free sets in type~I group.

\begin{lemma}\label{lem-type1-group-stability}
	Suppose that $G$ is type $I(p)$ and write $p = 3k + 2$. Let $A \subseteq G$ be sum-free of size $|A|>\left(\frac{1}{3}+\frac{1}{3(p+1)}\right)n$. Then there exists a homomorphism $\phi: G \rightarrow \bZ/p\bZ$  such that $A$ is contained in $\phi^{-1}({k+ 1,...,2k+ 1})$.
\end{lemma}

We will need the following simple fact about abelian groups.

\begin{fact}\label{fact-order2}
	Let $G:=\bZ_{2^{\alpha_1}}\oplus \cdots\oplus\bZ_{2^{\alpha_r}}\oplus \bZ_{p_1^{\beta_1}}\oplus \cdots\oplus \bZ_{p_t^{\beta_t}}$ and $g\in G$. Then there are at most $2^r$ solutions in $G$ to the equation $2x=g$.
\end{fact}

We will also use the classical result of Hardy and Ramanujan on asymptotics of the partition function.
\begin{theorem}\label{thm-partition}
	The number of partitions of integer $n$ is asymptotically $$\frac{1}{4n\sqrt{3}}\exp\left(\pi\sqrt{\frac{2n}{3}}\right).$$
\end{theorem}


\subsection{Maximal independent sets in graphs}\label{sec-graph-aux}
In this subsection we collect together  results on maximal independent sets in a graph. Let $\mis(\Ga)$ denote the number of maximal independent sets in a graph $\Ga$. 

Moon and Moser~\cite{MM} showed that for any graph $\Ga$, 
\begin{align}\label{eq-mis-MM}
	\mis(\Ga)\le 3^{|\Ga|/3},
\end{align} 
and this bound is optimal for disjoint union of triangles. When a graph is almost regular and relatively dense, the bound above can be improved as follows (Equation (3) in~\cite{BLST}).
\begin{lemma} \label{lem-mis-regular-dense}
	Let $k\ge 1$ and let $\Ga$ be a graph on $n$ vertices. 
	Suppose that $\De(\Ga)\le k\de(\Ga)$ and set $b:= \sqrt{\de(\Ga)}$.  Then
	\begin{align*}
	\mis(\Ga)\le \sum_{0 \le i\le n/b}{n\choose i}3^{{\left(\frac{k}{k+1}\right)\frac{n}{3}} + \frac{2n}{3b}}  .
	\end{align*}
\end{lemma}
When a graph is triangle-free, Hujter and Tuza~\cite{HT} obtained the following exponential improvement, which is optimal witnessed by a perfect matching. If $\Ga$ is triangle-free, then
\begin{align}\label{htnew}
	\mis(\Ga)\le 2^{|\Ga|/2}.
\end{align}
We will also make use of the following version for `almost triangle-free' graphs (Corollary~3.3 in~\cite{BLST2}). 

\begin{lemma}\label{lem-mis-almosttrifree}
	Let $n,D \in \bN$ and $k \in \mathbb{R}$. Suppose that $\Ga$ is a graph and $T$ is a set such that $\Ga':=\Ga\setminus T$ is triangle-free.
	Suppose that $\De (\Ga) \leq D$, $v(\Ga')=n$  and $e (\Ga') \geq n/2+k$. Then 
	$$\mis(\Ga) \leq 2^{n/2-k/(100D^2)+2|T|}.$$
\end{lemma}

%
%
%
%


\section{Proof of Theorem~\ref{thm-evengroups}}\label{sec-almostallevens}
To state Theorem~\ref{thm-evengroups} formally, we use the following proposition.
\begin{proposition}\label{prop-almostall}
	For any $\ep>0$, there exists $n_0>0$ such that the following holds for all $n>n_0$. Let $\bZ_{2^{\alpha_1}}\oplus\ldots\oplus\bZ_{2^{\alpha_r}}\oplus\bZ_{p_1^{\beta_1}}\oplus\ldots\oplus\bZ_{p_t^{\beta_t}}$ be the canonical decomposition of an abelian group of order $n$. Then all but $\ep$-proportion of abelian groups of order $n$ satisfy $2^r\le \ep n$.
\end{proposition}
\begin{proof}
	Fix $\ep>0$, and let  $n$ be sufficiently large. Let $\alpha,h\in\bN$ be such that $n=2^\alpha\cdot h$ and $2\nmid h$. So $\alpha=\sum_{i\in[r]}\alpha_i\ge r$. We may assume that $2^\alpha=n/h>\ep n$, as otherwise all order-$n$ groups have the desired property. 
	
	We first bound the number of groups with $2^r> \ep n$. As 
	$\prod_{i\in[t]}p_i^{\beta_i}=h<1/\ep$, there are only $O_{\ep}(1)$ ways to choose the odd components $\bZ_{p_i^{\beta_i}}$. Similarly, as	
	\begin{align*}
		\prod_{i=1}^{r}2^{\alpha_i-1}=\frac{2^{\alpha}}{2^r}<\frac{1}{\ep},
	\end{align*}
	the number of possibilities for $\alpha_i\ge 2$, i.e.~the non-$\bZ_2$ even components, is at most $O_{\ep}(1)$. 
	
	On the other hand, the number of non-isomorphic abelian groups of order $n$  is at least the number of partitions of $\alpha$, which, by Theorem~\ref{thm-partition}, is at least $2^{\sqrt{\alpha}}$ (as $\alpha>\log(\ep n)$ is sufficiently large).
\end{proof}
We can now restate Theorem~\ref{thm-evengroups} as follows.
\begin{theorem}\label{thm-evengroupsequiv}
	There exists a constant $c>10^{-4}$, such that the following holds for sufficiently large $n$. Given an abelian group $G=\bZ_{2^{\alpha_1}}\oplus \cdots\oplus\bZ_{2^{\alpha_r}}\oplus \bZ_{p_1^{\beta_1}}\oplus \cdots\oplus \bZ_{p_t^{\beta_t}}$
	with even order $n$, if $2^r=o(n)$, then,
	\begin{align*}
		f_{\max}(G)\le 2^{(1/2-c)\mu(G)}.
	\end{align*}
	
\end{theorem}


\subsection{Connection between sum-free sets and independent sets}
In this subsection, we will reduce the problem of estimating $\fm(G)$ to bounding $\mis(\Ga)$, for some Cayley-like graph $\Ga$. For subsets $B, S\subseteq G$, let $L_S[B]$ be the \emph{link graph of $S$ on $B$} defined as follows. The vertex set of $L_S[B]$ is $B$. The edge set of $L_S[B]$ consists of the following two types of edges:

\begin{enumerate}[leftmargin=*]
	\item[(i)] Two vertices $x$ and $y$ are adjacent if there exists an element $z\in S$ such that $\{x,y,z\}$ is a Schur triple; 
	\item[(ii)] There is a  loop at a vertex $x$ if $\{x,x, z\}$ or $\{x, z,z'\}$ is a Schur triple for some $z, z'\in S$.
\end{enumerate}

We will use the following container theorem of Green and Ruzsa (Proposition $2.1$ in~\cite{GR-g}). See also~\cite{BMS-container} for more on container method.
\begin{lemma}\label{lem-container}
	For all finite abelian group $G$, there is a family $\cF$ of subsets of $G$ with the following properties.
	\begin{enumerate}
		\item\label{it-fewsums} Every $F\in \cF$ has at most $(\log n)^{-1/9}n^2$ Schur triples.
		\item\label{it-containsall} If $S\subseteq G$ is sum-free, then $S$ is contained in some $F\in \cF$.
		\item\label{it-fewsets} $|\cF|=2^{n(\log n)^{-1/18}}$.
		\item\label{it-small} Every member of $\cF$ has size at most $\mu(G)+n(\log n)^{-1/50}$.
	\end{enumerate}
\end{lemma}
Note that the last property is not mentioned in~\cite{GR-g}. It follows form property(\ref{it-fewsums}) and supersaturation of sum-free sets in abelian groups (Proposition~$2.2$ in~\cite{GR-g}).

We refer to the sets in $\cF$ as \emph{containers}. For the rest of this section, fix an arbitrary group $G$  of order $n$ that satisfies the hypothesis of Theorem~\ref{thm-evengroupsequiv}. Let $F\in \cF$ be an arbitrary container. Recall that since $G$ is an even order group, $\mu(G)=n/2$. Thus, by Lemma~\ref{lem-container}~(\ref{it-containsall}) and~(\ref{it-fewsets}), to prove Theorem~\ref{thm-evengroupsequiv}, it suffices to show that, there exists a constant $c>0$ such that $f_{\max}(F) \leq 2^{(1/4-c)n}$, where $f_{\max}(F)$ denotes the number of maximal sum-free subsets of $G$ that lie in $F$.

By a group removal lemma of Green (Theorem~$1.4$ in~\cite{G-R}, see also~\cite{ksv}), $F=B\cup C$, where $B$ is sum-free and $|C|=o(n)$.  Notice that every maximal sum-free subset of $G$ in $F$ can be built in the following two steps: 
\begin{enumerate}
	\item[(1)] Choose a sum-free set $S$ in $C$; 
	\item[(2)] Extend $S$ in $B$ to a maximal one.
\end{enumerate}
Since the set $C$ is small, the number of choices for the first step is negligible. We will use the following lemma from~\cite{BLST} to bound the number of choices in the second step.
\begin{lemma}[\cite{BLST}]\label{lem-mis}
	Suppose that $B,S \subseteq G$ are both sum-free. If $I\subseteq B$ is  such that $S\cup I$ is a maximal sum-free subset of $G$, then $I$ is a maximal independent set in $L_S[B]$.
\end{lemma}


 For a fixed $S$, by Lemma~\ref{lem-mis}, the number of extensions of $S$ in $B$ in Step (2) is at most $\mis(L_S[B])$. Thus,
\begin{align*}
\fm(F)\le 2^{o(n)}\cdot \max_{\substack{S\subseteq C\\ S\text{ is sum-free}}}\mis(L_S[B]).
\end{align*}
Therefore, it suffices to show that there exists a positive constant $c$ such that 
\begin{align}\label{eq-sufficient}
	\mis(L_S[B])\le 2^{(1/4-c)n}.
\end{align}

\subsection{A new bound for maximal independent sets of dense graphs}
We will use the following lemma to bound the number of maximal independent sets in the link graph. This lemma can be viewed as a stability version of Moon and Moser's bound~\eqref{eq-mis-MM}.
\begin{lemma}\label{lem-mis-cayley}
	Let $k\in \bZ$, $\Delta\in\bN$, and $C\ge 3^{\Delta/13}$. Let $\Ga$ be an $n$-vertex graph with $n+k$ edges and maximum degree $\Delta$, then 
	$$\mis(\Ga)\le C\cdot 3^{\frac{n}{3}-\frac{k}{13\Delta}}.$$
\end{lemma}

We remark that the lemma above is optimal up to the factor $13$ in the exponent. Indeed, consider the graph $H$ consisting of disjoint union of $K_4$'s. This graph is $3$-regular and has $3n/2$ edges and so $k=n/2$ with $\mis(H)=4^{n/4}$. Solving $4^{1/4}=3^{1/3-1/(6c)}$, we see that the constant $13$ cannot be improved to $c<9.32$.

We will use the following fact: write $N[v]:=N(v)\cup \{v\}$.
\begin{align}\label{eq-mis-nbrhd}
\mis(\Ga)\le \mis(\Ga\setminus\{v\})+\mis(\Ga\setminus N[v]).
\end{align}

\begin{proof}[Proof of Lemma~\ref{lem-mis-cayley}]
	We use induction on $n$. For the base case, suppose that $n\le 3$, then $k\le 0$. Thus, by~\eqref{eq-mis-MM}, the lemma trivially holds. 
	Now, let $\Ga$ be an $n$-vertex graph that satisfies the hypothesis of the lemma. If $k\le 0$, then the lemma trivially holds. Therefore, $e(\Ga)-n\ge 1$, and $\De\ge 3$. Also, if $k<\De^2$, by~\eqref{eq-mis-MM}, we have 
	\begin{align*}
	\mis(\Ga)\le 3^{\frac{n}{3}}= 3^{\frac{\Delta}{13}}\cdot 3^{\frac{n}{3}-\frac{\Delta}{13}}\le C\cdot 3^{\frac{n}{3}-\frac{k}{13\Delta}}.
	\end{align*}
	We may assume $k\ge \Delta^2$. Fix a vertex $v$ of degree $d$ with $3\le d\le \Delta$. Let $\Ga':=\Ga\setminus\{v\}$ with $\Delta':=\Delta(\Ga')$, $n':=v(\Ga')=n-1$ and
	$$e(\Ga')=(n-1)+(k-d+1)=:n'+k';$$ 
	and $\Ga'':=\Ga\setminus N[v]$ with $\Delta'':=\Delta(\Ga'')$, $n'':=n-d-1$ and 
	$$e(\Ga'')\ge n+k-d\Delta=(n-d-1)+(k-d\Delta+d+1)=:n''+k''.$$
	As $k\ge\Delta^2$, we have $k',k''>0$. By induction hypothesis on $\Ga'$ and $\Ga''$, and that $\Delta',\Delta''\le\Delta$, we get
	\begin{align*}
	\mis(\Ga')\le C\cdot 3^{\frac{n-1}{3}-\frac{k-d+1}{13\Delta' }}\le C\cdot 3^{\frac{n-1}{3}-\frac{k-d+1}{13\Delta }}=C\cdot 3^{\frac{n}{3}-\frac{k}{13\Delta }}\cdot 3^{-\frac{1}{3}+\frac{d-1}{13\Delta }},
	\end{align*}
	and
	\begin{align*}
	\mis(\Ga'')&\le C\cdot 3^{\frac{n-d-1}{3}-\frac{k-d\Delta +d+1}{13\Delta''}}\le C\cdot 3^{\frac{n}{3}-\frac{d+1}{3}-\frac{k-d\Delta +d+1}{13\Delta }}\\
	&= C\cdot 3^{\frac{n}{3}-\frac{k}{13\Delta }}\cdot 3^{-\frac{d+1}{3}+\frac{d}{13}-\frac{d+1}{13\Delta }}.
	\end{align*}
	This finishes the proof as $\mis(\Ga)\le \mis(\Ga')+\mis(\Ga'')$ due to~\eqref{eq-mis-nbrhd}, and 
	\begin{align*}
	3^{-\frac{d+1}{3}+\frac{d}{13}-\frac{d+1}{13\Delta }}+3^{-\frac{1}{3}+\frac{d-1}{13\Delta }}\le 0.9997,
	\end{align*}
	subject to $3\le d\le \Delta $.
\end{proof}

\subsection{Proof of~\eqref{eq-sufficient}}

We will use the following definitions and notations throughout the rest of this section. For disjoint sum-free subsets $A,A'\subseteq F$, we call an edge $xy\in E(L_A[A'])$ a \emph{type~1 edge} if $x-y=a$, for some $a\in A\cup (-A)$. Otherwise, if $x+y=a$, for some $a\in A$, we call it a \emph{type~2 edge}.  Denote by $E_1(L_A[A'])$ and $E_2(L_A[A'])$ the set of type $1$ and $2$ edges, respectively.  For $i\in [2]$ and $x\in A'$, let $e_i(L_A[A']):=|E_i(L_A[A'])|$ and $d_i(x,L_A[A'])$ be the number of type $i$ edges incident to $x$ in $L_A[A']$. We will omit $L_A[A']$ from the above notations whenever clear from the context. 

Let $\Gamma:=L_S[B]$. Recall that $\mu(G)=n/2$ and $B$ is sum-free, therefore, $v(\Ga)\le n/2$. Denote by $\Ga_1$ and $\Ga_2$ the subgraph of $\Ga$ consisting of type $1$ and $2$ edges, respectively.

We claim that we may assume $B$ is a coset of some index $2$ subgroup of $G$. Indeed, if $|B|\le 4n/9$, then by~\eqref{eq-mis-MM}, we have 
\begin{equation}\label{eq-smallB}
	\mis(\Gamma)\le 3^{|B|/3}=3^{4n/27}<2^{0.24 n},
\end{equation}
as desired. Therefore, we can assume that $|B|>4n/9$, which, together with Lemma~\ref{lem-type1-group-stability}, implies that $B\subseteq G$ consists of elements with odd values in exactly one of the first $r$ coordinates, and $S\subseteq G$ is a subset of the set of elements with even values in the same coordinate. Note that adding new vertices to a graph will not decrease the number of maximal independent sets, and thus without loss of generality we can assume that
\begin{align*}
B=\{1,3,5,7,\ldots,2^{\alpha_{1}}-1\}\oplus \bZ_{2^{\alpha_2}}\oplus \cdots\oplus\bZ_{2^{\alpha_r}}\oplus \bZ_{p_1^{\beta_1}}\oplus \cdots\oplus \bZ_{p_t^{\beta_t}};
\end{align*}
and the generating set is 
\begin{align*}
S\subseteq \{0,2,4,6\ldots,2^{\alpha_{1}}-2\}\oplus \bZ_{2^{\alpha_2}}\oplus \cdots\oplus\bZ_{2^{\alpha_r}}\oplus \bZ_{p_1^{\beta_1}}\oplus \cdots\oplus \bZ_{p_t^{\beta_t}}.
\end{align*}
We first show that~\eqref{eq-sufficient} holds for  large $S$. We will use the following claim, which bounds the degree of every vertex, as well as, their type~$1$ and type~$2$ degree.
\begin{claim}\label{cl-degree-evengroup}
	For all $x\in B$,
\begin{enumerate}
	\item[(i)] $d_1(x,\Ga)=|S\cup(-S)|$,
	\item[(ii)] $d_2(x,\Ga)\le|S|$ with equality if and only if  $2x\notin S+(S\cup (-S))$.
\end{enumerate}
	Consequesntly,
\begin{enumerate}
	\item[(iii)] $|S\cup (-S)|\le\de(\Ga)\le\De(\Ga)\le |S\cup (-S)|+|S|\le 2\de(\Ga).$  
\end{enumerate}
\end{claim}
\begin{proof}
	Fix an arbitrary element $x\in B$. For the first part, note that each $s\in S\cup (-S)$ generates a unique type~$1$ neighbour $y=x+s\in B$. For the second part, it is clear that $x$ is incident to exactly $|S|$ many edges of the form $(x,s-x)$. Fix an $s\in S$. It suffices to show that the edge $(x, s-x)\in E_2(L_S[B])$ if and only if $2x\notin s+(S\cup (-S))$. By definition, $(x,s-x)\in E_2(L_S[B])$, if and only if $(x,s-x)\notin E_1(L_S[B])$. This is equivalent to $s-x\neq x-s'$, for all $s'\in S\cup (-S)$. Which implies $2x\notin s+(S\cup (-S))$.
\end{proof}

\begin{claim}\label{cl-largeS}
	If $|S|\ge 10^4$, then $\mis(\Ga)\le 2^{0.24n}$.
\end{claim} 	
\begin{proof}	
	Applying Lemma~\ref{lem-mis-regular-dense} with $k=2$ and $b=100$, we get
	\begin{align*}
	\mis(\Ga)\le \sum_{0 \le i\le n/200}{n/2\choose i}\cdot 3^{\frac{n}{9} + \frac{n}{300}}\le 2\cdot (100e)^{\frac{n}{200}}\cdot 3^{\frac{n}{9} + \frac{n}{300}}\le 2^{0.24n}.
	\end{align*}
\end{proof}

Thus, we may assume that $|S|<10^4$. The next claim will bound the density of $\Ga$ with size of $S$.
\begin{claim}\label{cl-edges-fewevens}
	We have
	\begin{align*}
		e(\Ga)\ge \frac{(|S\cup (-S)|+|S|)\cdot |B|}{2}-|S|\cdot |S\cup (-S)|\cdot 2^r\ge \frac{\De(\Ga)\cdot |B|}{2}-|S|\cdot |S\cup (-S)|\cdot 2^r.
	\end{align*}
\end{claim}
\begin{proof}
	By Claim~\ref{cl-degree-evengroup}(i), we only need to prove
	\begin{align*}
	e_2(\Ga)\ge \frac{|S|\cdot |B|}{2}-|S|\cdot|S\cup (-S)|\cdot 2^r.
	\end{align*}
	Define
	\begin{gather*}
		A:=\{(x,s-x): x\in B \text{ and } s\in S\}.
	\end{gather*}
	By definition $E_2\subseteq A\subseteq E$, and $e_2(\Ga)=|A|-|A\cap E_1|$. Then, it suffices to prove that $|A\cap E_1|\le|S|\cdot|S\cup (-S)|\cdot 2^r$. Recall that an edge $(x, s-x)\in A$ is a type~$1$ edge when $x-(s-x)=s'$, for some $s'\in S\cup (-S)$. Therefore, $|A\cap E_1|$ is at most the number of triples $(x,s,s')$ with $x\in B$, $s\in S$, $s'\in S\cup (-S)$, and $2x=s+s'$. By Fact~\ref{fact-order2}, the number of such triples is at most $|S|\cdot |S\cup(-S)|\cdot 2^r$,  yielding the desired bound.
\end{proof}

\case{1: $S=\{s\}$.} Let $\ell$ be the order of $s$. Next claim shows that if $s$ is not of order $3$, then $\Ga$ can be made triangle-free by removing $o(n)$ vertices.

\begin{claim}\label{cl-allevens-1gen-fewtri}
	If $\ell\neq 3$, then there exists a subset $B_{t}\subseteq B$ with $|B_t|\le 2^r=o(n)$ that intersects all triangles in $\Ga$, i.e. $V(T)\cap B_t\neq \emptyset$, for all triangles $T\subseteq \Ga$. 
\end{claim}
\begin{proof}
	Let $T$ be a triangle in $\Ga$ with $V(T)=\{x,y,z\}$. First, we will show that $E(T)$ 
	contains exactly one type $2$ edge. Indeed, if $|E(T)\cap E_2|\ge 2$, say $xy, xz\in E_2$, then $y=s-x=z$, a contradiction. Otherwise, if $|E(T)\cap E_2|=0$, then it is not hard to check that either $x-y=s$, $y-z=s$, and $z-x=s$, in which case, we have $3s=0$, or $x-y=s$ and $x-z=s$, in which case $y=z$, leading to a contradiction in either case.
	
	Therefore, without loss of generality, we can assume that $xy,xz\in E_1$ and $yz\in E_2$. We also assume that $x-y=s$ (the case $y-x=s$ is almost identical). Then, we must have $z-x=s$, which, together with $y+z=s$, implies $2z=3s$. By Fact~\ref{fact-order2}, since $s$ is fixed, there are at most $2^r$ choices for $z$. Then $B_t:=\{z:2z=3s\}$ is the desired set.     
\end{proof}

Let $B_t$ be the set guaranteed by Claim~\ref{cl-allevens-1gen-fewtri}. Let $\Ga'=\Ga\setminus B_t$. By Claims~\ref{cl-degree-evengroup} and~\ref{cl-edges-fewevens}, all vertices in $\Ga$ have degree at most $3|S|=3$ and $e(\Ga)\ge 3|B|/2-2^{r+1}$, and thus, 
$$
e(\Ga')\ge e(\Ga)-3|B_t|\ge \frac{3|B|}{2}-2^{r+3}.
$$
Recall that $2^r=o(n)$, thus by Lemma~\ref{lem-mis-almosttrifree}, we get
\begin{align*}
\mis(\Ga)\le 2^{\frac{\mu(G)}{2}-\frac{\mu(G)}{900}+o(n)}\le  2^{0.499\mu(G)},
\end{align*}
as desired.
Therefore, we can assume that $\ell=3$. In this case,  $\Ga_1$ is a disjoint union of $|B|/3$ triangles, where the vertex set of each triangle is $\{x,x+s,x+2s\}$, for some $x\in B$. Let $s=(2a,b)$, where $a\in \bZ_{2^{\alpha_1-1}}$ and $b\in \bZ_{2^{\alpha_2}}\oplus \cdots\oplus\bZ_{2^{\alpha_r}}\oplus \bZ_{p_1^{\beta_1}}\oplus \cdots\oplus \bZ_{p_t^{\beta_t}}$. For type~$2$ edges, notice that there exists a loop at all vertices $x\in B$ such that $2x=s$. By Claim~\ref{cl-degree-evengroup}, we have that for all vertices $x\in B$ such that $2x\neq 0$ and $2x\neq 2s$, we have $d_2(x)=1$. We call a vertex $x\in B$ \emph{irregular} if $2x\in \{0,s,2s\}$. Then if in a triangle of $\Ga_1$ one of the vertices is irregular, the other two are irregular as well. Additionally, one of the vertices has a loop in $\Ga_2$. Therefore, by Fact~\ref{fact-order2}, there are at most $2^r$ irregular triangles, i.e. triangles with all irregular vertices. Denote by $\cT'$ and $\cT$ the set of irregular and all triangles in $\Ga$, respectively. Then it is not hard to see that there exists a partition of all triangles in  $\cT\setminus\cT'$ into pairs, $\{x,x+s,x+2s\}$ and $\{s-x,-x,-x-s\}$, for $x\in B$, such that $\Ga_2$ induces a perfect matching between each pair. Since the number of maximal independent sets for two disjoint triangles joined by a perfect matching is $6$, and the number of maximal independent sets for a triangle that contains a vertex with a loop is two, we obtain that 
\begin{align*}
\mis(\Ga)\le 6^{\frac{|\cT\setminus\cT'|}{2}}\cdot  2^{|\cT'|}\le 6^{\frac{n/6-|\cT'|}{2}}\cdot  2^{|\cT'|}\le  6^{\frac{n}{12}}\le 2^{0.45\mu(G)}, 
\end{align*}
which finishes the proof of~\eqref{eq-sufficient} for the case when $S$ is a singleton.

\case{2: $|S|=2$, and $S=\{s,-s\}$.}  In this case, $|S\cup (-S)|=2$, which together with Claims~\ref{cl-degree-evengroup} and~\ref{cl-edges-fewevens}, implies $\de(\Ga)\ge 2$, $\De(\Ga)\le 4$, and $e(\Ga)= n-4\cdot 2^r$. Therefore, we can apply Lemma~\ref{lem-mis-cayley} with $C=3^{4/13}$ and $k=n/2-4\cdot 2^r$, which shows that
\begin{align*}
\mis(\Ga)\le 3^{4/13}\cdot3^{\frac{n}{6}-\frac{n/2-4\cdot 2^r}{13\De}}\le 3^{\left(\frac{1}{6}-\frac{1}{104}+o(1)\right)n}\le 2^{0.2491n}.
\end{align*}

\case{3: $3\le |S|\le  10000$ or $S=\{s_1,s_2\}$  with $s_1\neq -s_2$.} In this case, $|S\cup (-S)|\ge 4$. Also, by Claim~\ref{cl-degree-evengroup}, we have that $4\le\De(\Ga)\le 3|S|$. Furthermore, by Claim~\ref{cl-edges-fewevens},
\begin{align*}
e(\Ga)\ge \frac{\De(\Ga)\cdot n}{4}-o(n).
\end{align*}
Therefore, we can apply Lemma~\ref{lem-mis-cayley} with $k=(\De(\Ga)-2)n/4-o(n)$ to get
\begin{align*}
\mis(\Ga)\le 3^{\frac{\De}{13}}\cdot 3^{\frac{n}{6}-\frac{n(\De-2)}{52\De}+o(n)}\le 3^{\frac{\De}{13}}\cdot 3^{\left(\frac{1}{6}-\frac{1}{52}+\frac{1}{104}+o(1)\right)n}\le 2^{0.2491n}.
\end{align*}




\section{Other type~I groups}\label{sec-type1}
We can extend the proof for Theorem~\ref{thm-evengroups} to some other type I groups. We streamline the proof of Theorem~\ref{thm-other-type1} in this section.
\subsection{Group $G=\bZ_2^t\oplus \bZ_4$}\label{sec-z4} 
We may again assume, by Lemma~\ref{lem-type1-group-stability}, that $B$ is a a coset of a subgroup of index $2$, as otherwise $B$ is of small size, and we can apply~\eqref{eq-mis-MM} as in~\eqref{eq-smallB} to obtain the desired bound. We split the proof into the following two cases. 

\case{1:}$B= \bZ_2^{t}\oplus\{1,3\}$ and $S\subseteq \bZ_2^{t}\oplus\{0,2\}$.  We may assume that $\Ga$ does not contain any loops. Indeed, a loop at a vertex $x\in B$ implies that $2x=s$ for some $s\in S$. But then every vertex in $B$ has a loop, as $2B=\{(0,\ldots,0,2)\}$.
For all $s\in S$ and $A\subseteq S$, define $\overline{s}:=(0,\ldots,0,2)+s$, $\overline{A}:=\cup_{s\in A} \overline{s}$, and $A^*=A\cup \overline{A}$.

\begin{lemma}\label{lem-case1-regular}
	The graph $\Ga$ is $|S^*|$-regular.
\end{lemma}
\begin{proof}
	Note first that all elements in $S^*$ have order 2, so $S^*\cup (-(S^*))=S^*$. Then Claim~\ref{cl-degree-evengroup} implies that in $L_{S^*}[B]$ all vertices are adjacent to exactly $|S^*|$ many type~$1$ edges. It suffices to show that $xy\in E(\Ga)$ if and only if $xy\in E_1(L_{S^*} [B])$. ($\Rightarrow$) If $xy\in E_1(\Ga)$, it is trivial as $S\subseteq S^*$. Otherwise, if $xy\in E_2(\Ga)$, then there exists an $s\in S$ such that $s-x=y$. This, together with $2x=s-\overline{s}$, implies that $x+\overline{s}=s-x=y$, i.e.~$xy\in E_1(L_{S^*}[B])$. ($\Leftarrow$) Let $xy\in E_1(L_{S^*}[B])$. We may assume that $x-y=\overline{s}$ for some $\overline{s}\in \overline{S}$. Then  $x+y=x-y+2y=\overline{s}+(0,\ldots,0,2)=s$, that is, $xy\in E_2(\Ga)$ as claimed. 
\end{proof}

If $|S^*|\ge 4$, then we can apply Lemma~\ref{lem-mis-cayley} to get the desired bound. Suppose that $|S^*|\le 3$. Then it must be that $S^*=\{s,\overline{s}\}$. We claim that $\Ga$ is triangle-free, which together with $e(\Ga)=\mu(G)$ and Lemma~\ref{lem-mis-almosttrifree} yields the desired bound. Indeed, a triangle $T$ in $\Ga$ must have $V(T)=\{x,x+s,x+\overline{s}\}$, for some $x\in B$. Then we have $(x+s)+\overline{s}=x+\overline{s}$, or $s=0$, a contradiction as otherwise every vertex has a loop.

\case{2:}$B= \{1\}\oplus \bZ_2^{t-1}\oplus \bZ_4$ and $S\subseteq \{0\}\oplus \bZ_2^{t-1}\oplus \bZ_4$. Define $B_0= \{1\}\oplus \bZ_2^{t-1}\oplus \{0,2\}$ and $B_1= \{1\}\oplus \bZ_2^{t-1}\oplus\{1,3\}$. We  partition $S=S_0\cup S_1$ such that $S_0\subseteq \{0\}\oplus \bZ_2^{t-1}\oplus \{0,2\}$ and $S_1\subseteq \{0\}\oplus \bZ_2^{t-1}\oplus \{1,3\}$.

We may assume that $\Gamma$ does not contain any loop. Similar to Case 1, since $2B_i=\{(0,\ldots,0,2i)\}$, if there is one loop on a vertex $x\in B_i$, then every vertex in $B_i$ would have a loop, and by~\eqref{eq-mis-MM}, we have $\mis(\Ga)\le 3^{\frac{\mu(G)/2}{3}}\le 2^{0.27 \mu(G)}$.

Note that all edges in $E(\Ga[B_0])\cup E(\Ga[B_1])$ and $E(\Ga[B_0,B_1])$ are generated by $S_0$ and $S_1$ respectively.

\begin{lemma}\label{lem-case2-regular}
	For all $i\in \{0,1\}$ and $x_i\in B_i$
	\begin{itemize}
		\item $N_{\Ga[B_0]}(x_0)=x_0+S_0$, 
		
		\item $N_{\Ga[B_1]}(x_1)=x_1+S^*_0$, 
		
		\item $N_{\Ga[B_0,B_1]}(x_i)=x_i+S^*_1$. 
	\end{itemize} 
	
	In particular, $\Ga[B_0]$, $\Ga[B_1]$, and $\Ga[B_0,B_1]$ are $|S_0|$, $|S_0^*|$, and $|S_1^*|$-regular, respectively. Furthermore, $\Ga[B_0]$ is triangle-free. 
\end{lemma}
\begin{proof}
	Recall that edges in $\Ga[B_0]$ are generated by $S_0$. As all elements in $B_0\cup S_0$ are of order 2, for any $x\in B_0$ and $s_0\in S_0$, all three edges incident to $x$ generated by $s_0$, $\{x, x+s_0\}$, $\{x, x-s_0\}$ and $\{x, s_0-x\}$, coincide, showing that $\Ga[B_0]$ is $|S_0|$-regular. To see that $\Ga[B_0]$ is triangle-free, assume to the contrary that there exists a triangle $T\subseteq \Ga[B_0]$ with $V(T)=\{x_0,x_0+s_0,x_0+s'_0\}$, for some $s_0,s'_0\in S_0$. Then $x_0+s_0+s''_0=x_0+s'_0$ for some $s''_0\in S_0$. This implies that $s_0+s_0''=s'_0$, contradicting to $S_0$ being sum-free. The proof for $\Ga[B_1]$ being $|S_0^*|$-regular is almost identical to that of Lemma~\ref{lem-case1-regular}.

	For the bipartite graph $\Ga[B_0,B_1]$, all 
	edges are generated by $S_1$. Note that there is no type 2 edges, since elements in $B_0$ have order 2 and so $\{x,s_1-x\}$ coincides with $\{x, s_1+x\}$ for any $x\in B_0$ and $s_1\in S_1$. Thus, all edges are of the form $x\pm s_1$, showing that $\Ga[B_0,B_1]$ is $|S_1^*|$-regular as $S_1^*=S_1\cup (-S_1)$ due to $\overline{s_1}=-s_1$. 
\end{proof}

An immediate consequence is that the link graph is relatively regular: $|S|\le \de(\Ga)\le \De(\Ga)\le 2\de(\Ga)$. We may then assume that $|S|\le 20000$, as otherwise it can be handled as in Claim~\ref{cl-largeS}.

Suppose that $S_1=\emptyset$. Then $\Ga$ is a disjoint union of $\Ga[B_i]$, $i\in\{0,1\}$. By~\eqref{htnew} and Lemma~\ref{lem-case2-regular}, $\mis(\Ga[B_0])\le 2^{\mu(G)/4}$. It suffices to show that $\mis(\Ga[B_1])$ is exponentially smaller than $2^{\mu(G)/4}$. Recall that $\Ga[B_1]$ is $|S_0^*|$-regular, then similar analysis as in Case~1 implies the desired bound.

We may now assume that $|S_1|\ge 1$. Furthermore, $|S_0|\ge 1$, as otherwise $\Ga=\Ga[B_0,B_1]$ is a $D$-regular bipartite graph with $D\ge 2$ and Lemma~\ref{lem-mis-almosttrifree} implies the desired bound. 

Define $d_0:=|S_0|+|S^*_1|$ and $d_1:=|S^*_0|+|S^*_1|$. By Lemma~\ref{lem-case2-regular}, all vertices in $B_0$ and $B_1$ have degree $d_0$ and $d_1$, respectively. Note that $d_0\le d_1\le 2d_0$. Thus,
$e(\Ga)=\frac{\mu(G)}{2}\left(\frac{d_0}{2}+\frac{d_1}{2}\right)$.
Hence, Lemma~\ref{lem-mis-cayley}, together with  $\De(\Ga)=d_1$, implies
\begin{align*}
\mis(\Ga)\le 3^{\frac{d_1}{13}}\cdot 3^{\frac{\mu(G)}{3}-\frac{\mu(G)}{4}\cdot\frac{d_0+d_1-4}{13d_1}},
\end{align*}
which, by a short calculation, is exponentially smaller than $2^{\mu(G)/2}$ when $d_0\ge 4$. We can then assume
$$d_0=|S_0|+|S_1\cup (-S_1)|\le 3.$$
As $S_0$ and $S_1$ are non-empty and elements in $S_1$ have order $4$, we must have $S_0=\{s_0\}$ and $S_1=\{s_1\}$, in which case $\Ga[B_0],\Ga[B_1],\Ga[B_0,B_1]$ are $1$-, $2$- and $2$-regular respectively. We claim that $\Ga$ is triangle-free. Then Lemma~\ref{lem-mis-almosttrifree}, together with $e(\Ga)=7\mu(G)/4$, implies the desired bound.

Suppose to the contrary that there exists a triangle $T$. As $\Ga[B_0]$ is triangle-free, $V(T)\cap B_1\neq \emptyset$. If $V(T)\subseteq B_1$, then $V(T)=\{x,x+s_0,x+\overline{s_0}\}$ and $N_{\Ga[B_1]}(x+s_0)=\{x,x+s_0+\overline{s_0}\}$, implying that $s_0=0$, a contradiction. If $V(T)$ intersects $B_1$ at two vertices, then we must have $V(T)=\{x_0, x_0+s_1,x_0-s_1\}$ for some $x_0\in B_0$. This, however, implies that either $(x_0+s_1)+s_0=x_0-s_1$ or $(x_0+s_1)+\overline{s_0}=x_0-s_1$. The former case implies that $2s_1=s_0$; while the latter case yields $s_0=0$, leading to contradictions in both cases. The case when $V(T)$ intersects $B_0$ at two vertices can be handled similarly.







\subsection{$G=\bZ_5\oplus H$ and $2\nmid |H|$}\label{sec-z5}

In this section, we prove that $\mis(\Ga)$ is exponentially smaller than $2^{\mu(G)/2}=2^{n/5}$.  In particular, we will show that there exists a positive constant $c$,    
\begin{align}\label{eq-thmz5}
\mis(\Ga)\le 2^{(1/2-c)\mu(G)}.
\end{align}

If $B$ is smaller than $0.37n$, then~\eqref{eq-mis-MM} suffices. Note that for type I(5) groups, the stability Lemma~\ref{lem-type1-group-stability} applies only to sets of size at least $7n/18\approx 0.389n$, nonetheless with the same proof in~\cite{GR-g}, the stability can be slightly strenghten to cover sets of size at least $11n/30\approx 0.367n$. We may then assume that $B=\{2,3\}\oplus H$ and $S\subseteq \{0,1,4\}\oplus H$. For all subsets $G'\subseteq G$, denote $G'_i$ to be the set $G'\cap \{i\}\oplus H$, for all $i\in \bZ_5$.

Similar to Lemma~\ref{lem-case2-regular}, the following claim on neighbourhoods of vertices in $\Ga$ can be derived. We omit its proof.

\begin{claim}\label{cor-z5-edges}
	For all $i\in\{2,3\}$ and  $x_i\in B_i$,
	\begin{itemize}
		\item $d_1(x_i,\Ga[B_i])=|S_0\cup(-S_0)|$;
		
		\item $d_2(x_i,\Ga[B_i])=|S_{2i}|-|\{s\in S_{2i}: 2x_i\in s+(S_0\cup (-S_0))\}|$;
		
		\item $d_1(x_i,\Ga[B_2,B_3])=|S_4\cup(-S_1)|$;
		
		\item $d_2(x_i,\Ga[B_2,B_3])=|S_0|-|\{s\in S_0: 2x_i\in s+(S_{2i}\cup (-S_{-2i}))\}|$.
	\end{itemize}
\end{claim}

An immediate consequence is that the link graph is relatively regular: $|S|/2\le \de(\Ga)\le \De(\Ga)\le 2\de(\Ga)$. We may again assume that $|S|=O(1)$, as otherwise it can be handled as in Claim~\ref{cl-largeS}. As now $\De(\Ga)=O(1)$, we can make use of the following corollary of Lemma~\ref{lem-mis-cayley}.

\begin{claim}\label{cl-z5-alde}
	If $e(\Ga)\ge (1+\alpha)|B|-O_S(1)$, and $\alpha/\De(\Ga)\ge 1/4$, then $\Ga$ satisfies~\eqref{eq-thmz5}.
\end{claim}

For the rest of the proof, without loss of generality, assume that $|S_1|\ge |S_4|$. Next, we will calculate the ratio $\alpha/\De(\Ga)$ depending on size of $S$. By Claim~\ref{cor-z5-edges}, 
\begin{align*}
e(\Ga)= \frac{|B|}{4}\cdot\left(2|S_0\cup (-S_0)|+|S_4|+|S_1|+2|S_4\cup (-S_1)|+2|S_0|\right)-O_S(1),
\end{align*}
and
\begin{align*}
\De(\Ga)=|S_0\cup (-S_0)|+|S_1|+|S_0|+|S_4\cup (-S_1)|.
\end{align*}
Therefore,
\begin{align*}
\frac{\alpha}{\De(\Ga)}
=\frac{1}{4}\cdot\frac{2|S_0\cup (-S_0)|+|S_4|+|S_1|+2|S_4\cup (-S_1)|+2|S_0|-4}{|S_0\cup (-S_0)|+|S_1|+|S_0|+|S_4\cup (-S_1)|}.
\end{align*}
By Claim~\ref{cl-z5-alde}, we may assume $\alpha/\De(\Ga)<1/4$, implying that 
\begin{align}\label{eq-z5-suff}
|S_0\cup (-S_0)|+|S_0|+|S_4\cup (-S_1)|+|S_4|\le 3.
\end{align}
In particular, we must have $|S_0|\le 1$. 

Suppose that $|S_0|=1$. As $H$ has no order-2 element, $|S_0\cup (-S_0)|=2$ and hence $S_1, S_4=\emptyset$. By Claim~\ref{cor-z5-edges},  $\Ga[B_2,B_3]$ is a matching and 
apart from $O_S(1)$ vertices, $\Ga[B_i]$, $i\in\{2,3\}$, is a disjoint union of $\ell$-cycle, where $\ell$ is the order of $s\in S_0$. If $\ell\neq 3$, then $\Ga$ is triangle-free. Note that $\De(\Ga)= 3$ and $e(\Ga)=3\mu(\Ga)/2-O_S(1)$, then Lemma~\ref{lem-mis-almosttrifree} finishes the proof of this case. If $\ell=3$, then apart from constantly many vertices, $\Ga$ is a disjoint union of the six-vertex graph obtained by adding a perfect matching between two triangles. Thus $\mis(\Ga)\le 6^{(1/6+o(1))\mu(G)}\le 2^{0.45\mu(G)}$.

We may then assume that $S_0=\emptyset$. Note that $S_1\neq\emptyset$, as otherwise $S=\emptyset$. Thus $e(\Ga)\ge 3|S_1||B|/4\ge 3|B|/4$. We shall see that in this case $\Ga$ can be made triangle-free by removing constantly many vertices. Then as $e(\Ga)\ge 3|B|/4$, Lemma~\ref{lem-mis-almosttrifree} finishes the proof. Recall that $H$ contains no element of order 2, thus, it suffices to show every triangle contains a vertex $y$ with $2y\in S+S-S$ as $|S|=O(1)$. Let $T$ be a triangle induced by $\{x,y,z\}$. Assume that $x\in B_2$ and $y,z\in B_3$, other cases are similar. Recall that edges in $[B_2,B_3]$ and $B_3$ are type 1 and type 2 respectively. Then $x+s'=y$, $x+s''=z$ and $y+z=s$ for some $s,s',s''\in S_1$, implying that $2y=s+s'-s''$ as desired.


\section{Type II groups}\label{sec-type2}

\begin{proof}[Proof of Proposition~\ref{prop-9}]
	Upper bound follows from~\eqref{eq-g}. For the lower bound, as $9||G|$, either $G=\bZ_{3^a}\oplus G'$ with $a\ge 2$, or $G=\bZ_3\oplus \bZ_3\oplus G'$. In the former case, let $H< \bZ_{3^a}$  be a subgroup of index three. Then, $B:=(1+H) \oplus G'$ is a sum-free subset of size $\mu(G)$. Since $a\ge 2$, we have that $3||H|$. Let $x\in H$ be of order three in $H$, and define $s:=(x,0_{G'})$. Note that $s$ has order three in $G$. Note that the graph $L_{\{s\}}[B]$ does not have any type~$2$ edges. Indeed, for all $(1+y, z)\in B$, with $y\in H$ and $z\in G'$, the element $s-(1+y, z)\in (2+H) \oplus G'$. Therefore, every vertex in $L_{\{s\}}[B]$ has degree exactly two. Since $s$ has order three, it is easy to check that $L_{\{s\}}[B]$ is a disjoint union of triangles $T$ with
	\begin{align*}
	V(T)=\{(1+y,z),(1+y+x,z),(1+y+2x,z)\}, \text{ for }y\in H\text{ and }z\in G'.
	\end{align*}
	
	Suppose now $G=\bZ_{3}\oplus \bZ_{3}\oplus G'$. Then, the set $B:=\{1\} \oplus \bZ_{3}\oplus G'$ is a sum-free subset of size $\mu(G)$. Let $s:=(0,1,0_{G'})$, and similar to the previous case, there are no type~$2$ edges in $L_{\{s\}}[B]$, and also $L_{\{s\}}[B]$ is a disjoint union of triangles $T$ with 
	\begin{align*}
	V(T)=\{(1,y,z),(1,y+1,z),(1,y+2,z)\}, \text{ for some }y\in \bZ_{3}\text{ and }z\in G'.
	\end{align*}
	
	Thus, in either case, the link graph is a disjoint union of triangles. Note that every maximal independent set $I$ in $L_{\{s\}}[B]$ corresponds naturally to a maximal sum-free set containing $I\cup \{s\}$ in $G$, and thus, $\fm(G)\ge\mis(L_{\{s\}}[B])=3^{|B|/3}$, as desired.
\end{proof}

\section{Concluding remarks}\label{sec-rmk}
In this paper, we show that type II groups of order divisible by 9 have many maximal sum-free sets,~$3^{(1/3+o(1))\mu(G)}$; while almost all even order group, i.e.~type I(2), have exponentially fewer than $2^{\mu(G)/2}$. This is in sharp contrast to the integers setting. Many interesting problems remain. For example, very little is known about type III groups. We conclude this paper with two further remarks.

\medskip

\begin{itemize}
	\item We establish the bound $\fm(G)\le 2^{(1/2-c)\mu(G)}$ for even order groups with sublinear number of order 2 elements. New ideas are needed to handle the remaining constant many even order groups with $\Omega(n)$ number of order 2 elements. We see in Section~\ref{sec-z4} that the same bound holds for the group $\bZ_2\oplus \cdots \oplus \bZ_{2}\oplus \bZ_{4}$, which is in a sense the `worst' even order group as it has the most number of order 2 elements (other than $\bZ_2^k$). Considering also the result on type I(5) groups $\bZ_5\oplus H$, it is plausible that $\bZ_2^k$ is the group with the most number of maximal sum-free sets among type I groups.
	\begin{conjecture}
		All type I groups $G$ except $\bZ_2^k$ have exponentially fewer maximal sum-free sets than $2^{\mu(G)/2}$.
	\end{conjecture}
	Apart from the even order groups with many order 2 elements, another difficulty for the above conjecture is that the stability result gets weaker for type I($p$) groups when $p$ gets larger. As a result, we might not be able to assume the ground set of the link graph is a union of cosets, which is very useful in our analysis.
	
   \medskip
	
	\item The remaining type II groups not covered by Proposition~\ref{prop-9} are of the form $G=\bZ_3\oplus_i\bZ_{p_i^{a_i}}$ with $p_i\equiv 1$ (mod $3$). It is known~\cite{BLST2} that $$2^{\mu(G)/2}\le \fm(G)\le 3^{(1/3+o(1))\mu(G)}.$$
	It would be interesting to know which bound is closer to the truth.
\end{itemize}


\bigskip

{\footnotesize \obeylines \parindent=0pt
Mathematics Institute, University of Warwick, UK.
}

\begin{flushleft}
	{\it{E-mail addresses}:
		\tt{ $\lbrace$h.liu.9,~m.sharifzadeh$\rbrace$@warwick.ac.uk}}
\end{flushleft}

\end{document}